\documentclass[11pt,twoside]{article}

\usepackage{amsmath}
\usepackage{mathrsfs}
\usepackage{amssymb}
\usepackage{latexsym}
\usepackage{amsthm}
\usepackage{fancyhdr}
\usepackage{titlesec}

\usepackage[mathscr]{eucal}

\usepackage[all]{xy}
\usepackage{pb-diagram,pb-xy}
\usepackage[a4paper,top=3cm, bottom=3cm, inner=3.5cm,outer=2.7cm,nofoot,headsep=1.5cm]{geometry}
\numberwithin{equation}{section}

\titleformat{\section}{\Large\centering\scshape}{\thesection}{1em}{}

\titleformat{\subsection}{\large\bfseries}{\thesubsection}{1em}{}

\newtheoremstyle{alex}%
{20pt}%
{6pt}%
{\itshape}%
{}%
{\bfseries}%
{.}%
{ }%
{}%

\newtheoremstyle{nadja}%
{20pt}%
{12pt}%
{}%
{}%
{\bfseries}%
{.}%
{ }%
{}%

\theoremstyle{alex}
\newtheorem{thm}{Theorem}[section]
\newtheorem{lem}[thm]{Lemma}
\newtheorem{prop}[thm]{Proposition}
\newtheorem{cor}[thm]{Corollary}
\newtheorem{fac}[thm]{Fact}
\theoremstyle{nadja}
\newtheorem{defn}[thm]{Definition}

\newtheorem{rem}[thm]{Remark}
\theoremstyle{plain}
\newtheorem*{cla}{Claim}

\newcommand{\K}{\mathbb{K}}
\newcommand{\F}{\mathbb{F}}

\newcommand{\M}{\mathcal{M}}

\newcommand{\N}{\mathbb{N}}

\newcommand{\dcl}{\operatorname{dcl}}
\newcommand{\alg}{\operatorname{alg}}

\newcommand{\Gal}{\operatorname{Gal}}
\newcommand{\Th}{\operatorname{Th}}

\author{Nadja Hempel\thanks{Partially supported by ANR-09-BLAN-0047 Modig and ANR-13-BS01-0006-01 ValCoMo}}

\title{On $n$-dependent groups and fields}

\AtEndDocument{\bigskip{\footnotesize%
  \textsc{Universit\'e de Lyon; CNRS; Universit\'e Lyon 1; Institut Camille Jordan UMR5208, 43 boulevard du 11 novembre 1918, F-69622 Villeurbanne Cedex, France} \par  
  \textit{E-mail address}: \texttt{hempel@math.univ-lyon1.fr} \par
}}

\begin{document}

\maketitle
\begin{abstract}
First, an example of  a $2$-dependent group without a minimal subgroup of bounded index is given. Second,
all infinite $n$-dependent fields are shown to be Artin-Schreier closed.  Furthermore, the theory of any  non separably closed PAC field has the IP$_n$ property for all natural numbers $n$ and certain properties of dependent (NIP) valued fields extend to the $n$-dependent context.
\end{abstract}

\pagestyle{fancy}
\fancyhead{}
\fancyfoot{}
\fancyhead[LE]{\thepage}
\fancyhead[CE]{\textsc{\small{Nadja Hempel}}}
\fancyhead[RO]{\thepage}
\fancyhead[CO]{\textsc{\small{On $n$-dependent groups and fields}}}

\section{Introduction}\label{introductionsection}

Macintyre \cite{am_omega} and Cherlin-Shelah \cite{gc_ss} have shown independently that any superstable field is algebraically closed. However, less is known in the case of supersimple fields. Hrushovski proved that any infinite perfect bounded pseudo-algebraically closed (PAC) field is supersimple \cite{ehPAC} and conversely supersimple fields are perfect and bounded (Pillay and Poizat \cite{ap_bp}), and it is conjectured that they are PAC.
More is known about Artin-Schreier extensions of certain fields. Using a suitable chain condition for uniformly definable subgroups, Kaplan, Scanlon and Wagner showed in \cite{ik_ts_fow} that infinite NIP fields of positive characteristic are Artin-Schreier closed and simple fields have only finitely many Artin-Schreier extensions. The latter result was generalized to fields of positive characteristic defined in a theory without the tree property of the second kind (NTP$_2$ fields) by Chernikov, Kaplan and Simon \cite{ac_ik_ps}.

We study groups and fields without the $n$-independence property. Theories without the $n$-independence property, briefly $n$-dependent or NIP$_n$ theories, were introduced by Shelah in \cite{Shstrongdep}. They are a natural generalization of NIP theories, and in fact both notions coincide when $n$ equals to $1$. For background on NIP theories the reader may consult \cite{pslecnotes}. It is easy to see that any theory with the $(n+1)$-independence property has the $n$-independence property. On the other hand, as for any natural number $n$ the random $(n+1)$-hypergraph is ${n+1}$-dependent but has the $n$-independence property \cite[Example 2.2.2]{CPT}, the classes of $n$-dependent theories form a proper hierarchy. Additionally, since all random hypergraphs are simple, the previous example shows that there are theories which are simple and $n$-dependent but which are not NIP. Hence one might ask if there are any non combinatorial examples of n-dependent theories which have the independence property? And furthermore, which results of NIP theories can be generalized to $n$-dependent theories or more specifically which results of (super)stable theories remains true for (super)simple $n$-dependent theories? Beyarslan \cite{OeBe}  constructed the random $n$-hypergraph in any pseudo-finite field or, more generally, in any e-free perfect PAC field (PAC fields whose absolute Galois group is the profinite completion of the free group on $e$ generators). Thus, those fields lie outside of the hierarchy of $n$-dependent fields.

In this paper, we first give an example of a group with a simple 2-dependent theory which has the independence property. Additionally, in this group the $A$-connected component depends on the parameter set $A$. This establishes on the one hand a non combinatorial example of a proper 2-dependent theory and on the other hand shows that the existence of an absolute connected component in any NIP group cannot be generalized to 2-dependent groups. Secondly, we find a Baldwin-Saxl condition for $n$-dependent groups (Section \ref{sec_cc}). Using this and connectivity of a certain vector group established in Section \ref{sec_svg} we deduce that $n$-dependent fields are Artin-Schreier closed (Section \ref{sec_ASE}). Furthermore, we show in Section \ref{sec_PAC} that the theory of any non separably closed PAC field has in fact the IP$_n$ property for all natural numbers $n$ which was established by Duret for the case $n$ equals to $1$  \cite{jld}. In Section \ref{sec_app} we extend certain consequences found in \cite{ik_ts_fow} for dependent valued fields with perfect residue field as well as  in \cite{fjjk} by Jahnke and Koenigsmann for NIP henselian valued field to the $n$-dependent context.

I would like to thank my supervisors Thomas Blossier and Frank O. Wagner for useful comments during the work on this article and on first versions of this paper. Also, I like to thank Artem Chernikov for bringing this problem to my attention and to Daniel Palac\'in for valuable discussions around the topic.

\section{Preliminaries}
In this section we introduce $n$-dependent theories and state some general facts. The following definition can be found in \cite[Definition 2.4]{Sh2dep}.

\begin{defn}
Let $T$ be a theory. We say that a formula $\psi(\bar y_0, \dots ,\bar y_{n-1}; \bar x)$ in $T$  has the \emph{$n$-independence property} (IP$_n$) if there exists some parameters $(\bar a_i^j : i \in \omega, j \in n)$ and $(\bar b_I : I \subset \omega^n)$ in some model $\M$ of $T$ such that $\M \models \psi(\bar a_{i_0}^0, \dots , \bar a_{i_{n-1}}^{n-1}, \bar b_I)$ if and only if $(i_0, \dots, i_{n-1}) \in I$.

A theory is said to have IP$_n$ if one of its formulas has IP$_n$. Otherwise we called it \emph{$n$-dependent}. A structure is said to have IP$_n$ or to be $n$-dependent if its theory does.
\end{defn}

Both facts below are useful in order to proof that a theory is $n$-dependent as it reduces the complexity of formulas one has to consider to have IP$_n$.
The first one is stated as Remark 2.5  \cite{Sh2dep} and afterwards proved in detail as Theorem 6.4 \cite{CPT}.
\begin{fac}\label{fact_x1}
A theory $T$ is $n$-dependent if and only if every formula $\phi(\bar y_0 , . . . , \bar y_{n-1}; x )$ with $|x| = 1$ is $n$-dependent.
\end{fac}

\begin{fac}\cite[Corollary 3.15]{CPT} \label{fac_CloBooCom}
Let $\phi(\bar y_0 , . . . , \bar y_{n-1} ; \bar x)$ and $\psi(\bar y_0 , . . . , \bar y_{n-1 };\bar x) $ be $n$-dependent
formulas. Then so are $\neg \phi$, $\phi \wedge \psi$ and $\phi \vee \psi$.
\end{fac}

\begin{rem}\label{rem_IPn}
Note that a formula with at most $n$ free variables cannot witness the $n$-independence property. Thus, from the previous fact  it is easy to deduce that the random $n$-hypergraph is $n$-dependent. In fact, more generally any theory in which any formula of more than $n$ free variables is a boolean combination of formulas with at most $n$ free variables is $n$-dependent.
\end{rem}

\section{Example of a $2$-dependent group without a minimal subgroup of bounded index}

Let $G$ be $\F_p^{(\omega)}$ where $\F_p$ is the finite field with $p$ elements. We consider the structure $\M$ defined as $(G, \F_p, 0, +,\cdot)$ where $0$ is the neutral element, $+$ is addition in $G$, and $\cdot$ is the bilinear form $(a_i)_i \cdot (b_i)_i = \sum_i a_i b_i$ from $G$ to $\F_p$. This example in the case $p$ equals $2$ has been studied by Wagner in \cite[Example 4.1.14]{Wag}. He shows that it is simple and that the connected component $G^0_A$ for any parameter set $A$ is equal to $\{g \in G : \bigcap_{a \in A} g \cdot a = 0\}$. Hence, it is getting smaller and smaller while enlarging $A$ and whence the absolute connected component, which exists in any NIP group, does not for this example.

\begin{lem}\label{lem_MEQ}
The theory of $\M$ eliminates quantifiers.
\end{lem}

\begin{proof}
Let $ t_1(x;\bar y )$ and $ t_2(x;\bar y )$ be two group terms in $G$ and let $ \epsilon$ be an element of $\F_p$.
Observe that the atomic formula $t_1(x;\bar y) = t_2(x;\bar y)$ (resp.  $t_1(x;\bar y) \neq  t_2(x;\bar y)$)  is equivalent to an atomic formula of the form $x = t(\bar y)$ or $ 0 = t(\bar y)$ (resp.  $x \neq  t(\bar y)$ or $ 0 \neq t(\bar y)$) for some group term $t(\bar y)$. Note that $ 0 = t(\bar y)$ as well as $ 0 \neq t(\bar y)$ are both quantifier free formulas in the free variables $\bar y $. Furthermore, the atomic formulas   $t_1(x;\bar y) \cdot t_2(x;\bar y) = \epsilon$ and $ t_1(x;\bar y )\cdot t_2(x;\bar y )\not= \epsilon$ are equivalent to a boolean combination of atomic formulas of the form $x\cdot x = \epsilon_x$, $x\cdot t_i(\bar y) = \epsilon_i$ and $t_j(\bar y)\cdot t_k(\bar y) = \epsilon_{jk}$ (a quantifier free formula in the free variables $\bar y$)  with $t_i(\bar y )$ group terms and $\epsilon_x$, $\epsilon_i$, and  $\epsilon_{jk}$ elements of $\F_p$.
Thus, a quantifier free formula $\varphi(x, \bar y)$ is equivalent to a finite disjunction of formulas of the form
\[\phi(x; \bar y)= \psi(\bar y) \wedge x \cdot x = \epsilon \wedge \bigwedge_{i \in I_0} x = t_i^0(\bar y ) \wedge  \bigwedge_{i \in I_1} x \neq t_i^1(\bar y) \wedge \bigwedge_{i \in I_2} x \cdot t_i^2(\bar y )= \epsilon_i \]
where $ t_i^j(\bar y )$ are group terms, $\epsilon, \epsilon_i$ are elements of $\F_p$, and $\psi(\bar y)$ is a quantifier free formula in the free variables $\bar y$. If $I_0$ is nonempty, the formula $\exists x \phi(x, \bar y)$ is equivalent to
\[\psi(\bar y) \wedge  \bigwedge_{j, l \in I_0} t_j^0(\bar y ) = t_l^0(\bar y ) \wedge t_i^0(\bar y )  \cdot  t_i^0(\bar y )  = \epsilon \wedge \bigwedge_{j \in I_1}  t_i^0(\bar y )  \neq t_j^1(\bar y) \wedge \bigwedge_{j \in I_2}  t_i^0(\bar y )  \cdot t_j^2(\bar y )= \epsilon_j \]
for any $i \in I_0$. Now, we assume that $I_0$ is the empty set. If there exists an element $x'$ such that $x' \cdot z_i = \epsilon_i$ for given $z_0, \dots, z_m$ in $G$ and $\epsilon_i \in \F_p$, one can always find an element $x$ such that $ x\cdot x= \epsilon$ and $x \neq v_j$ for given $v_0, \dots, v_q$ in $G$ which still satisfies  $x \cdot z_i = \epsilon_i$  by modifying $x'$ at a large enough coordinate. Hence, it is enough to find a quantifier free condition which is equivalent to $\exists x \bigwedge_{i \in I_2} x \cdot t_i^2(\bar y )= \epsilon_i $. For $i \in \F_p$, let
\[Y_i=\{j\in I_2:\epsilon_j=i\}.\]
Then $\exists x \bigwedge_{i \in I_2} x \cdot t_i^2(\bar y )= \epsilon_i $ is equivalent to
\[\bigwedge_{i=0}^{p-1} \bigwedge_{j \in Y_i} t_j^2(\bar y) \notin \left\{ \sum_{k \in Y_0} \lambda_k^0 t_k^2(\bar y ) + \dots + \sum_{k \in Y_{i}\setminus j} \lambda_k^{i} t_k^2(\bar y ): \lambda_k^l \in \F_p, \sum_{l=1}^{i} \sum_{k\in Y_l} ^{k \neq j}l\cdot_{\F_p} \lambda_k^l  \neq i \right\}\]
which finishes the proof.
\end{proof}

\begin{lem}\label{FactEss}
The structure $\M$ is 2-dependent.
\end{lem}

\begin{proof}
We suppose, towards a contradiction, that $\M$ has IP$_2$.  By Fact \ref{fact_x1} we can find a formula  $\phi(\bar y_0, \bar y_1; x)$ with $| x|=1$ which witnesses the 2-independence property. By the proof of Lemma \ref{lem_MEQ} and as being $2$-dependent is preserved under boolean combinations (Fact \ref{fac_CloBooCom}),  it suffices to prove that none of the following formulas can witness the $2$-independence property in the variables $(\bar y_0, \bar y_1;x)$:
\begin{itemize}
\item quantifier free formulas of the form $ \psi(\bar y_0, \bar y_1)$,
\item the formula $ x\cdot x = \epsilon $ with $\epsilon$ in $\F_p$,
\item formulas of the form $ x = t( \bar y_0, \bar y_1 )$ for some group term $ t( \bar y_0, \bar y_1 )$,
\item formulas of the form $  x \cdot t( \bar y_0, \bar y_1 ) = \epsilon$ for some group term $ t( \bar y_0, \bar y_1 )$ and $\epsilon$ in $\F_p$.
\end{itemize}
As the atomic formula  $ \psi(\bar y_0, \bar y_1)$ does not depend on $x$ and $ x\cdot x = \epsilon $  does not depend on $\bar y_0$ nor $\bar y_1$ they cannot witness the $2$-independence property in the variables  $(\bar y_0, \bar y_1;x)$. Furthermore, as for given $\bar a$ and $\bar b$, the formula  $ x = t( \bar a, \bar b )$ can be only satisfied by a single element, such a formula is as well $2$-dependent. Thus the only candidate left is a formula of the form $ x \cdot t(\bar y_0, \bar y_1 ) = \epsilon$ with $ t(\bar y_0, \bar y_1 )$ some group term in $G$ and $\epsilon$ an element of $\F_p$. Thus, we suppose that the formula  $ x \cdot t(\bar y_0, \bar y_1 ) = \epsilon$ has IP$_2$ and choose some elements $\{ \bar a_i: i \in \omega\}$, $\{ \bar b_i: i \in \omega\}$ and $\{c_I: I \subset \omega^2\}$ which witnesses it.  As $t(\bar y_0, \bar y_1 )$ is just a sum of elements of the tuple $\bar y_0$ and $\bar y_1$ and $G$ is commutative, we may write this formula as $ x \cdot (t_a(\bar y_0) +t_b( \bar y_1 )) = \epsilon$
in which the term $t_a(\bar y_0)$ (resp. $t_b(\bar y_1)$) is a sum of elements of the tuple $\bar y_0$ (resp. $\bar y_1$). Let
\[ S_{ij}:= \{x :  x \cdot (t_a(\bar a_i )+ t_b(\bar b_j )) = \epsilon\}\]
be the set of realizations of the formula $x \cdot (t_a(\bar a_i )+ t_b(\bar b_j )) = \epsilon$.
Note, that an element $c$ belongs to $ S_{ij}$  if and only if we have that $e_{ij} (c)$ defined as
\[e_{ij} (c) =  c  \cdot \left(t_a (\bar a_i) + t_b( \bar b_j)\right)\]
is equal to $ \epsilon$.
 Let $i$, $ l$, $ j$, and $k$ be arbitrary natural numbers. Then,
\begin{align*}
e_{ij} (c)& =  c  \cdot \left(t_a (\bar a_i) + t_b( \bar b_j)\right) \\
& =  c  \cdot \left((t_a (\bar a_i) + t_b( \bar b_k)) + (p-1)( t_a (\bar a_l) + t_b( \bar b_k))+ (t_a (\bar a_l) + t_b( \bar b_j))\right) \\
& =   e_{ik} (c)+(p-1) e_{lk}(c) +e_{lj} (c).
\end{align*}
If the element $c$ belongs to $ S_{ik} \cap  S_{lk} \cap  S_{lj}$, the terms $e_{ik}(c)$, $e_{lk}(c)$, and $e_{lj} (c)$ are all equal to $\epsilon$. By the equality above we get that $e_{ij} (c)$ is also equal to $\epsilon$ and so $c$ also belongs to $S_{ij}$.

Let $I = \{ (1,1),(1,2),(2,2)\}$. Then $c_I \in  S_{22} \cap  S_{12} \cap  S_{11}$ but $c_I \not\in S_{21}$ which contradicts the precious paragraph letting $i$ and $k$ be equal to $2$ and $l$ and $j$ be equal to $1$. Thus the formula $ x \cdot t(\bar y_0, \bar y_1 ) = \epsilon$ is $2$-dependent, hence all formulas in the theory of $\M$ are $2$-dependent and whence $\M$ is $2$-dependent.
\end{proof}

\section{Baldwin-Saxl condition for $n$-dependent theories}\label{sec_cc}

We shall now prove a suitable version of the Baldwin-Saxl condition \cite{BS} for $n$-dependent formulas.

\begin{prop}\label{prop_cc} Let $G$ be a group and let $\psi(\bar y_0,\ldots,\bar y_{n-1};x)$ be a $n$-dependent formula for which the set $\psi(\bar b_0,\ldots,\bar b_{n-1};G )$ defines a subgroup of $G$ for any parameters $\bar b_0, \dots,\bar  b_{n-1}$. Then there exists a natural number $m_\psi$ such that for any $d$ greater or equal to $m_\psi$ and any array of parameters $(\bar a_{i,j}: i <n,\ j \leq d)$ there is $\nu \in d^n$ such that
\[\bigcap_{\eta\in d^n} H_\eta=\bigcap_{\eta\in d^n, \eta\neq\nu} H_\eta\]
where $H_\eta$ is defined as $\psi(\bar a_{0, i_0},\ldots,\bar a_{n-1, i_{n-1}}; x)$ for $\eta = (i_0, \dots, i_{n-1})$.
\end{prop}

\begin{proof} Suppose, towards a contradiction, that for an arbitrarily large natural number $m$ one can find a finite array $(\bar a_{i,j}: i <n,\ j \leq m)$ of parameters such that $\bigcap_{\eta\in m^n} H_\eta$ is strictly contained in any of its proper subintersections. Hence, for every $\nu\in m^n$ there exists  $c_\nu$ in $  \bigcap_{\eta\neq\nu} H_\eta\setminus \bigcap_{\eta} H_\eta$.

Now, for any subset $J$ of $m^n$, we let $c_J:=\prod_{\eta\in J} c_\eta$. Note that $c_J\in H_\nu$ whenever $\nu \in m^n\setminus J$. On the other hand, if $\nu$ is an element of $ J$, all factors of the product except of $c_\nu$ belong to $H_\nu$, whence $c_J\not\in H_\nu$. By compactness, one can find an infinite array of parameters $(\bar a_{i,j}: i <n,\ j \leq \omega)$ and elements $\{c_J : J \subset \omega^n\}$ such that $c_J$ belongs to $H_\nu$ if and only if $\nu \not \in J$. Hence, the formula $\neg\psi(\bar y_0,\ldots,\bar y_{n-1}; x)$ has IP$_n$ and whence by Fact \ref{fac_CloBooCom} the original formula $\psi(\bar y_0,\ldots,\bar y_{n-1};x)$ has IP$_n$ as well contradicting the assumption.
\end{proof}

\section{A special vector group}\label{sec_svg}
For this section, we fix an algebraically closed field $\K$ of characteristic $p > 0$ and we let $\wp(x) $ be the additive homomorphism $x \mapsto x^p -x$ on $\K$.

We analyze the following algebraic subgroups of  $(\K, +)^{n}$:
\begin{defn}\label{def_Ga}
For a singleton $a$ in $\K$, we let $G_{a}$ be equal to $(\K, +)$, and for a tuple $\bar{a}=(a_0, \dots , a_{n-1}) \in \K^n$ with $n>1$ we define:
\[G_{\bar{a}} = \{ (x_0, \dots , x_{n-1}) \in \K^{n} :\ a_0 \cdot \wp(x_0)= a_i \cdot \wp(x_i)\mbox{ for }0 \leq i < n \}.\]
\end{defn}
Recall that for an algebraic group $G$, we denote by $G^0$ the connected component of the unit element of $G$. Note that if $G$ is definable over some parameter set $A$, its connected component $G^0$ coincides with the smallest $A$-definable subgroup of $G$ of finite index.

Our aim is to show that $G_{\bar{a}}$ is connected for certain choices of $\bar{a}$, namely  $G_{\bar{a}}$ coincides with $G^0_{\bar{a}}$.
\begin{lem}\label{lem_addpolyG}
Let $k$ be an algebraically closed subfield of $\K$, let $G$ be a $k$-definable connected algebraic subgroup of $(\K^n, +)$ and let $f$ be a $k$-definable homomorphism from $G$ to $(\K, +)$ such that for every $\bar g \in G$ there are polynomials $P_{\bar g}(X_0, \dots , X_{n-1})$ and $Q_{\bar g}(X_0, \dots , X_{n-1})$ in $k [X_0, \dots , X_{n-1}]$ such that
\[f(\bar g)=\frac{P_{\bar g}(\bar g)}{Q_{\bar g}(\bar g)}.\]Then $f$ is an additive polynomial in $k[X_0, \dots , X_{n-1}]$. In fact, there exists natural numbers $m_0, \dots, m_n$ such that $f$ is of the form $\sum_{i=0}^{m_0} a_{i,0} X_0^{p^i} + \dots + \sum_{i=0}^{m_n} a_{i,n} X_n^{p^i}$ with coefficients $a_{i,j}$ in $k$.
\end{lem}

\begin{proof}
By compactness, one can find finitely many definable subsets $D_i$ of $G$ and polynomials  $P_i(X_0, \dots , X_{n-1})$ and $Q_i(X_0, \dots , X_{n-1})$ in $k [X_0, \dots , X_{n-1}]$ such that $f$ is equal to $P_i(\bar x)/Q_i(\bar x)$ on $D_i$. Using \cite[Lemma 3.8]{tb} we can extend $f$ to a $k$-definable homomorphism $F: (\K^n, +)\rightarrow (\K,+)$ which is also locally rational. Now, the functions
\[F_0(X) := F(X, 0, \dots , 0), \ \dots \ ,F_{n-1}(X) := F(0, \dots , 0, X)\]
are $k$-definable homomorphisms of $(\K,+)$ to itself. Additionally, they are rational on a finite definable decomposition of $\K$, so they are rational on a cofinite subset of $\K$. Hence every $F_i$ is an additive polynomial in $k[X]$. Thus
\[F(X_0, \dots, X_{n-1}) = F_0(X_0) + \dots +F_{n-1}(X_{n-1})\]
is an additive polynomial in $k[X_0, \dots , X_{n-1}]$ as it is a sum of additive polynomials. By \cite[Proposition 1.1.5]{dg} it is of the desired form.
\end{proof}

\begin{lem}\label{lem_GaCon}
Let $\bar{a} = (a_0, \dots, a_n )$ be a tuple in $\K^{\times}$.  Then  $G_{\bar{a}}$ is connected if and only if the set $\left\{ \frac{1}{a_0}, \dots,   \frac{1}{a_n} \right\}$
is linearly $\F_p$-independent.

\end{lem}

Parts of the proof follows the one of \cite[Lemma 2.8]{ik_ts_fow}.
\begin{proof}

So suppose first that $\left\{ \frac{1}{a_0}, \dots,   \frac{1}{a_n} \right\}$
is linearly $\F_p$-dependent. Thus we can find elements $b_0, \dots , b_{n-1}$ in $\F_p$ such that
\[ b_0 \cdot  \frac{1}{a_0} + \dots + b_{n-1}  \frac{1}{a_{n-1}} =  \frac{1}{a_n} . \]
Now, let $\bar a '$ be the tuple $\bar a$ restricted to its first $n$ coordinates and fix some element $(x_0, \dots x_{n-1})$ in $G_{\bar a'}$. Let $t$ be defined as $ a_0 ( x_0^p -x_0)$. Hence, by the definition of $G_{\bar a'}$, we have that $t$ is equal to  $ a_i ( x_i^p -x_i)$ for any $ i < n$. Furthermore, we have that $(x_0, \dots, x_{n-1} , x)$ belongs to $G_{\bar a}$ if and only if
\begin{eqnarray*}
&t& = a_n (x^p - x)\\
 \Leftrightarrow&  0 &= \frac{1}{a_n} t  - (x^p-x) \\
 \Leftrightarrow & 0 & =\frac{b_0}{a_0}t + \dots +   \frac{b_{n-1}}{a_{n-1}} t  - (x^p-x) \\
 \Leftrightarrow  & 0 & = b_0 \cdot ( x_0^p -x_0)+ \dots + b_{n-1}\cdot  ( x_{n-1}^p -x_{n-1} ) -(x^p-x)\\
  \Leftrightarrow & 0 &= (b_0 \cdot x_0 + \dots +b_{n-1}\cdot x_{n-1} - x)^p - (b_0 \cdot x_0 + \dots + b_{n-1} x_{n-1} - x).
\end{eqnarray*}

In other words, $(x_0, \dots, x_{n-1} , x)$ belongs to $G_a$ if and only if $b_0 \cdot x_0 + \dots + b_{n-1} x_{n-1} - x$ is an element of $\F_p$. With this formulation we consider the following subset of $G_{\bar a}$:
\[ H = \{ (x_0, \dots x_n) \in G_{\bar a} : (x_0, \dots x_{n-1}) \in G_{\bar a '} \mbox{ and }  b_0 \cdot x_0 + \dots b_{n-1} x_{n-1} - x_n = 0\}\]
This is in fact a definable subgroup of $G_{\bar a}$ of finite index. Hence $G_{\bar a}$ is not connected.

We prove the other implication by induction on the length of the tuple $\bar{a}$ which we denote by $n$. Let $n=1$, then $G_{\bar{a}}$ is equal to $(\K, +)$ and thus connected since the additive group of an algebraically closed field is always connected.

Let $\bar{a}= (a_0, \dots, a_n)$ be an $(n+1)$-tuple such that $\left\{ \frac{1}{a_0}, \dots,   \frac{1}{a_n} \right\}$ is linearly $\F_p$-independent and suppose that the statement holds for tuples of length $n$. Define $\bar{a}'$ to be the restriction of $\bar{a}$ to the first $n$ coordinates. Observe that the natural map $\pi : G_{\bar{a}} \rightarrow G_{\bar{a}'}$ is surjective since $\K$ is algebraically closed and that
\[[G_{\bar{a}'} : \pi(G_{\bar{a}}^0)] = [\pi(G_{\bar{a}}) : \pi(G_{\bar{a}}^0)] \leq [G_{\bar{a}} : G_{\bar{a}}^0] < \infty.\]
Hence the definable group $\pi(G_{\bar{a}}^0)$ has finite index in $G_{\bar{a}'}$. As $\left\{ \frac{1}{a_0}, \dots,   \frac{1}{a_{n-1}} \right\}$ is also linearly $\F_p$-independent, the group $G_{\bar{a}'}$ is connected by assumption. Therefore $\pi(G_{\bar{a}}^0) = G_{\bar{a}'}$.

Now, suppose that $G_{\bar{a}}$ is not connected.
\\
\begin{cla} For every $\bar{x}\in G_{\bar{a}'}$, there exists a unique $x_n \in \K$ such that $(\bar{x}, x_n)\in G_{\bar{a}}^0$.
\end{cla}
\begin{proof}[Proof of the Claim] Assume there exists $\bar{x} \in \K^n$ and two distinct elements $x_n^0$ and $x_n^1$ of $\K$ such that $( \bar{x}, x_n^0)$ and $( \bar{x}, x_n^1)$ are elements of $G_{\bar{a}}^0$. As $G_{\bar{a}}^0$ is a group, their difference $( \bar{0}, x_n^0 - x_n^1)$ belongs also to $G_{\bar{a}}^0$. Thus, by definition of $G_{\bar{a}}$, its last coordinate $x_n^0 - x_n^1$ lies in $\F_p$. So $( \bar{0},\F_p)$ is a subgroup of  $G_{\bar{a}}^0$. Take an arbitrary element $( \bar{x}, x_n)$ in $G_{\bar{a}}$. As $\pi(G_{\bar{a}}^0) = G_{\bar{a}'}$, there exists $x_n' \in \K$ with $( \bar{x}, x_n') \in G_{\bar{a}}^0$. Again, the difference of the last coordinate $x_n' - x_n$ lies in $\F_p$. So
\[( \bar{x}, x_n) = (\bar{x}, x_n') - (\bar{0}, x_n' - x_n) \in G_{\bar{a}}^0.\]
This leads to a contradiction, as $G_{\bar{a}}^0$ is assumed to be a proper subgroup of $G_{\bar{a}}$.
\end{proof}

Thus, we can fix a definable additive function $f: G_{\bar{a}'} \rightarrow \K$ that sends every tuple to this unique element. Note that $G_{\bar{a}}$ and hence also $G_{\bar{a}}^0$ are defined over $\bar{a}$. So the function $f$ is defined over $\bar a$ as well. Now, let $\bar{x} = (x_0, \dots , x_{n-1})$ be any tuple in $G_{\bar{a}'}$ and set $L := \F_p(a_0, \dots, a_n)$. Then:
\[x_n := f(\bar{x}) \in \dcl(\bar{a}, \bar{x}).\]
In other words, $x_n$ is definable over $L(x_0, \dots, x_{n-1})$ which simply means that it belongs to the purely inseparable closure $\bigcup_{n \in \N} L(x_0, \dots, x_{n-1})^{p^{-n}}$ of $L(x_0, \dots , x_{n-1})$ by \cite[Chapter 4, Corollary 1.4]{eb}. Since there exists an $l \in L(x_0)$ such that $x_n^p-x_n-a_n^{-1}l= 0$, the element $x_n$ is separable over $L(x_0, \dots, x_{n-1})$. So it belongs to $L(x_0, \dots, x_{n-1})$ which implies that there exists some mutually prime polynomials $g , h \in L[X_0, \dots , X_{n-1}]$ such that $x_n = h(x_0, \dots , x_{n-1}) / g(x_0, \dots , x_{n-1})$. Thus, by Lemma \ref{lem_addpolyG} the definable function $f(X_0, \dots , X_{n-1})$ we started with  is an additive polynomial in $n$ variables over $L^{\alg}$ and there exists $c_{j,i}$ in $L^{\alg}$ and natural numbers $m_j$ such that
\[f(X_0, \dots , X_{n-1}) = \sum_{i=0}^{m_0} c_{0, i} X_0^{p^i} + \dots +  \sum_{i=0}^{m_{n-1}} c_{n-1,i} X_{n-1}^{p^i}.\]
Using the identities $X_i^p - X_i = \frac{a_0}{a_i} (X_0^p - X_0)$ in $G_{\bar{a}}^0$,  there are $\beta_j$ in $L^{\alg}$ and $g(X_0) = \sum_{i=1}^{m_0} d_{i} X_0^{p^i}$ an additive polynomial in $L^{\alg}[X_0]$ with summands of powers of $X_0$ greater or equal to $p$ such that
\[ f(X_0, \dots , X_{n-1}) = g(X_0) +  \sum_{j=0}^{n-1} \beta_j \cdot X_j.\]
Since the image under $f$ of the vectors $(0,1, 0, \dots, 0), (0,0,1, 0, \dots, 0), \dots , (0, \dots, 0,1) $ has to be an element of $\F_p$, for $0 < i < n$ the $\beta_i$'s have to be elements of $\F_p$. On the other hand, for any element $( x_0, \dots, x_n)$ of $G_{\bar{a}}^0$ we have that $a_n(x_n^p - x_n) = a_0(x_0^p - x_0)$. Replacing $x_n$ by $f(x_0, \dots , x_{n-1})$ we obtain
\begin{align*}
0 & = a_n  \left[f(x_0, \dots , x_{n-1})^p - f(x_0, \dots , x_{n-1})\right]  - a_0(x_0^p - x_0) \\
&= a_n \left[g(x_0)^p - g(x_0) + (\beta_0^p x_0^p - \beta_0 x_0)+ \sum_{j=1}^{n-1} \beta_j( x_j^p- x_j)\right]  - a_0(x_0^p - x_0) .
\end{align*}
Using again the identities $x_i^p - x_i = \frac{a_0}{a_i} (x_0^p - x_0)$ in $G_{\bar{a}}^0$ we obtain a polynomial in one variable
\[P(X) = a_n\left[g(X)^p - g(X) + (\beta_0^p X^p - \beta_0 X)+ \sum_{j=1}^{n-1} \beta_j \frac{a_0}{a_j} (X^p - X)\right]- a_0(X^p - X)\]
which vanishes for all elements $x_0$ of $\K$ such that there exists $x_1, \dots , x_{n-1}$ in $\K$ with $(x_0, \dots, x_{n-1}) \in G_{\bar{a}'}$. In fact, this is true for all elements of $\K$. Hence, $P$ is the zero polynomial. Notice that $g(X)$ appears in a $p$th-power. Since it contains only summands of power of $X$ greater or equal to $p$, the polynomial $g(X)^p$ contains only summands of power of $X$ strictly greater than $p$. As $X$ only appears in powers less or equal to $p$ in all other summands of $P$, the polynomial $g(X)$ has to be the zero polynomial itself. By the same argument as for the other $\beta_j$, the coefficient $\beta_0$ has to belong to $\F_p$ as well. Dividing by $a_0a_n$ yields that
\[\sum_{j=0}^{n} \beta_j \frac{1}{a_j} (X^p - X) \]
with $\beta_n := -1$ is the zero polynomial. Thus
\[\sum_{j=0}^{n} \beta_j \frac{1}{a_j}  =0\]
As $\beta_n$ is different from $0$ and all $\beta_i$ are elements of $\F_p$, this contradicts the assumption and the lemma is established.
\end{proof}

Using Lemma \ref{lem_GaCon}, a stronger version of \cite[Lemma 2.8]{ik_ts_fow} together with \cite[Corollary 2.6]{ik_ts_fow}, we obtain the following corollary in the same way as Kaplan, Scanlon and Wagner obtain \cite[Corollary 2.9]{ik_ts_fow}.

\begin{cor}\label{cor_ConComIso}
Let $k$ be a perfect subfield of $\K$ and $\bar{a} \in k^n$ be as in the previous lemma. Then $G_{\bar{a}}$ is isomorphic over $k$ to $(\K, +)$. In particular, for any field $K \geq k$ with $K \leq \K$, the group $G_{\bar{a}}(K)$ is isomorphic to $(K, +)$.
\end{cor}

\section{Artin-Schreier extensions}\label{sec_ASE}

\begin{defn}
Let $K$ be a field of characteristic $p>0$ and $\wp(x) $ the additive homomorphism $x \mapsto x^p -x$. A field extension $L / K$ is called an \emph{Artin-Schreier extension} if $L = K(a)$ with $\wp(a) \in K$. We say that $K$ is \emph{Artin-Schreier closed} if it has no proper Artin-Schreier extension i.\ e. $\wp(K) = K$.
\end{defn}

In the following remark, we produce elements from an algebraically independent array of size $m^n$ which fit the condition of Lemma \ref{lem_GaCon}.

\begin{rem}\label{rem_n-dep}
Let $\{\alpha _{i,j} : i \in n, j \in m\} $ be a set of algebraically independent elements in $\K$. Then the tuple $( a_{(i_0, \dots, i_{n-1})}: (i_0, \dots, i_{n-1}) \in m^n)$ with $a_{(i_0, \dots, i_{n-1})}=  \prod_{l=0}^{n-1}\alpha _{l, i_l}$ and ordered lexicographically satisfies the condition of Lemma \ref{lem_GaCon}.
\end{rem}
\begin{proof}
Suppose that there exists a tuple of elements $(\beta_{(i_0, \dots, i_{n-1})}:{(i_0, \dots, i_{n-1}) \in m^n})$ in $\F_p$ not all equal to zero such that
\[\sum_{ (i_0, \dots, i_{n-1}) \in m^n} \beta_{(i_0, \dots, i_{n-1})} \frac{1}{a_{(i_0, \dots, i_{n-1}} )} =0\]
Then the $\alpha_{i,j}$ satisfy:
\[\sum_{ (i_0, \dots, i_{n-1}) \in m^n} \beta_{(i_0, \dots, i_{n-1})} \cdot \left(\prod_{\{(k,l)\neq(j,i_j) : j \leq n-1 \} } \alpha_{k,l}  \right) =0\]
which contradicts the algebraic independence of the $\alpha_{i,j}$.
\end{proof}
We can now adapt the proof in \cite{ik_ts_fow} showing that an infinite NIP field is Artin-Schreier closed to obtain the same result for a $n$-dependent field.

\begin{thm}\label{thm_ASENIPn}
Any infinite $n$-dependent field is Artin-Schreier closed.
\end{thm}

\begin{proof}
Let $K$ be an infinite $n$-dependent field and we may assume that it is $\aleph_0$-saturated. We work in a big algebraically closed field $\K$ that contains all objects we will consider. Let $k = \bigcap_{l \in \omega} K^{p^l}$, which is a type-definable infinite perfect subfield of $K$. We consider the formula $\psi(x;y_0, \dots,y_{n-1}) := \exists t\ (x=\prod_{i=0}^{n-1}y_i \cdot \wp(t))$ which for every tuple $(a_0, \dots , a_{n-1})$ in $k^n$ defines an additive subgroup of $(K,+)$. Let $m \in \omega$ be the natural number given by Proposition \ref{prop_cc} for this formula. Now, we fix an array of size $m^n$ of algebraically independent elements $\{\alpha _{i,j} : i \in n, j \in m\} $ in $k$ and set $a_{(i_0, \dots, i_{n-1})}=  \prod_{l=0}^{n}\alpha _{l, i_l}$. By choice of $m$, there exists $(j_0, \dots, j_{n-1}) \in m^n$ such that
\begin{eqnarray}
\bigcap_{(i_0, \dots, i_{n-1}) \in m^n} a_{(i_0, \dots, i_{n-1})} \cdot \wp(K) = \bigcap_{(i_0, \dots, i_{n-1}) \neq (j_0, \dots, j_{n-1})} a_{(i_0, \dots, i_{n-1})} \cdot \wp(K).
\end{eqnarray}
By reordering the elements, we may assume that $(j_0, \dots, j_{n-1})= (m, \dots, m)$. Let $\bar{a}$ be the tuple $(a_{(i_0, \dots, i_{n-1})} : (i_0, \dots, i_{n-1}) \in m^n)$ ordered lexicographically and $\bar{a}'$ the restriction to $m^n -1$ coordinates (one coordinate less).

We consider the groups $G_{\bar{a}}$ and respectively $G_{\bar{a}'}$ defined as in Definition \ref{def_Ga}. Using Remark \ref{rem_n-dep} and Corollary \ref{cor_ConComIso} we obtain the following commuting diagram.
$$\xymatrix{
G_{\bar{a}} \ar[0,1]^{\pi} \ar[d]^\simeq
 & G_{\bar{a}'} \ar[d]^\simeq\\
(\K, +) \ar[0,1]^{\rho} & (\K,+)}$$
As the vertical isomorphisms are defined over $k$, this diagram can be restricted to $K$. Note that $\pi$ and therefore also $\rho$ stays onto for this restriction by equality (6.1) and that the size of $\ker(\rho)$ has to be $p$. Choose a nontrivial element $c$ in the kernel of $\rho$ and let $\rho'$ be equal to $\rho(c\cdot x)$. Observe that $\rho'$ is still a morphism from $(\K, +)$ to $(\K, +)$, its restriction to $K$ is still onto and its kernel is equal to $\F_p$. Then \cite[Remark 4.2]{ik_ts_fow} ensures that $\rho'$ is of the form $a \cdot (x^p - x)^{p^n}$ for some $a$ in $K$. Finally, let $l \in K$ be arbitrary. Since $\rho'\upharpoonright K$ is onto and $X^{p^n}$ is an inseparable polynomial in characteristic $p$, there exists $h \in K$ with $l = h^p-h$. As $l \in K$ was arbitrary, we get that $\wp(K) = K$ and we can conclude.
\end{proof}

The proof of \cite[Corollary 4.4]{ik_ts_fow}  adapts immediately and yields the following corollary.
\begin{cor}\label{cor_pnotdiv}
If $K$ is an infinite $n$-dependent field of characteristic $p>0$ and $L/K$ is a finite separable extension, then $p$ does not divide $[L:K]$.
\end{cor}

\section{Non separably closed PAC field}\label{sec_PAC}

The goal of this section is to generalize a result of Duret \cite{jld}, namely that the theory of a non separably closed PAC field has the IP property. To do so we need the following two facts.

\begin{fac}\label{facIP}\cite[Lemme 6.2]{jld}
Let $K$ be a field and $k$ be a subfield of $K$ which is PAC. Let $p$ be a prime number which does not coincide with the characteristic of $K$ such that $k$ contains all $p$th roots of unity and there exists an element in $k$ that does not have a $p$th root in $K$. Let $(a_i: i \in \omega)$ be a set of pairwise different elements of $k$ and let $I$ and $J$ be finite disjoint subsets of $\omega$, then $K$ realizes
\[\{ \exists y (y^p = x+ a_i):i \in I\} \cup \{\neg \exists y (y^p = x + a_j): j \in J \}.\]
\end{fac}

\begin{fac}\label{facPAC}\cite[Lemme 2.1]{jld}
Every finite separable extension of a PAC field is PAC.
\end{fac}

\begin{thm}
Let $K$ be a field and $k$ be a subfield of $K$ which is a non separably closed PAC field and relatively algebraically closed in $K$. Then, the theory of $K$ has the $n$-independence property.
\end{thm}

\proof
If $k$ is countable, we may work in an elementary extension of the tuple $(K,k)$ for which it is uncountable. As $k$ is non separably closed, there exists a proper Galois extension $l$ of $k$. Let $p$ be a prime number that divides the degree of $l$ over $k$. Then there is a separable extension $k'$ of $k$ such that the Galois extension $l$ over $k'$ is of degree $p$. We may distinguish two cases:

\begin{enumerate}
\item The characteristic of $k$ is equal to $p$. As $l$ is a cyclic Galois extension of degree $p$ of $k'$, a field of characteristic $p$, it is an Artin-Schreier extension of $k'$.  We pick $\alpha$ such that $k' = k (\alpha)$ and let $K' = K(\alpha)$. As $k'$ is separable over $k$, it is relatively algebraically closed in $K'$ by \cite[p.59]{La}. Hence $K'$ admits an Artin-Schreier extension and consequently its theory has IP$_n$ by Theorem \ref{thm_ASENIPn}. As it is an algebraic extension of $K$, thus interpretable in $K$, the theory $\Th(K)$ has IP$_n$ as well.

\item The characteristic of $k$ is different than $p$. Since $l$ is a separable extension of $k'$, we can find an element $\beta$ of $l$ such that $l$ is equal to $k'(\beta)$.
  Let $\omega$ be a primitive $p$-root of unity and let $k'_{\omega}= k'(\omega)$ and $l_{\omega}= l(\omega)$. Note that  $l_{\omega}$ is equal to $k'_{\omega}(\beta)$ and that the degree $[l_{\omega}:k'_{\omega}]$  is at most $p$ and the degree $ [k'_{\omega}: k']$ is strictly smaller than $p$. Additionally, we have:
\[[l_{\omega}:k'_{\omega}] \cdot [k'_{\omega}: k'] = [ l_{\omega}: k'] = [ l_{\omega}: l] \cdot [l: k'] = [ l_{\omega}: l] \cdot p.\]
Thus $[l_{\omega}:k'_{\omega}]$ is divisible by $p$ and hence equal to $p$. Furthermore, the conjugates of $\beta$ over $k'_{\omega}$ are the same as over $k'$. Hence, as $l$ is a Galois extension of $k'$, they are contained in $l$ and whence in $l_\omega$. Thus, the field $l_\omega$ is a cyclic Galois extension of the field $k'_\omega$ and $k'_\omega$ contains the $p$-roots of unity. In other words, $l_\omega$ is a Kummer extension of $k'_\omega$ of degree $p$. So there exists an element $\delta$ in $k'_\omega$ that does not have a $p$ root in it.
Furthermore, as $k'_\omega$ is a finite separable extension of $k$, it is also PAC by Fact \ref{facPAC} and it is relatively  algebraically closed in $K'_\omega = K'(\omega)$ by \cite[p.59]{La}. Thus, the element $\delta$ has no $p$-root in  $K'_\omega$  as well. Let $\{a_{i,j}: j < n, i \in \omega\}$ be a set of algebraic independent elements of $k'_\omega$ which exists as it is an uncountable field. This ensures that $\prod_{l=0}^{n-1} a_{i_l,l} \neq \prod_{l=0}^{n-1} a_{j_l,l}$ for $ (i_0, \dots, i_{n-1}) \neq  (j_0, \dots, j_{n-1})$. Thus we may apply Fact \ref{facIP} to $K'_\omega$, $k'_\omega$ and the infinite set $\{ \prod_{l=0}^{n-1} a_{i_l,l}: (i_0, \dots, i_{n-1}) \in \N^n\}$. We deduce that for the formula $\varphi(y;x_0, \dots, x_{n-1}) =  \exists z (z^p = y + \prod_{i=0}^{n-1} x_i)$ and for any disjoint finite subsets $I$ and $J$ of $\N^n$ there exists an element in $K'_\omega$ that realizes
\[ \{ \varphi(y; a_{i_0,0}, \dots, a_{i_{n-1},n-1})\}_{(i_0, \dots, i_{n-1}) \in I} \cup \{\neg \varphi(y; a_{j_0,0}, \dots, a_{j_{n-1},n-1})\}_{(j_0, \dots, j_{n-1}) \in J }\]
Thus $\Th(K'_\omega)$ has the IP$_n$ property by compactness. As again $K'_\omega$ is interpretable in $K$, we can conclude that the theory of $K$ has the IP$_n$ property as well.
\qed
\end{enumerate}

\begin{cor}
The theory of any non separably closed PAC field has the IP$_n$ property.
\end{cor}

In the special case of pseudo-finite fields or, more generally, e-free PAC fields the previous corollary is a consequence of a result of Beyarslan proved in \cite{OeBe}, namely that one can interpret the $n$-hypergraph in any such field.

\section{Applications to valued fields}\label{sec_app}

In \cite{ik_ts_fow} the authors deduce that an NIP valued field of positive characteristic $p$ has to be $p$-divisible simply by the fact that infinite NIP fields are Artin-Schreier closed \cite[Proposition 5.4]{ik_ts_fow}. Thus their result generalizes to our framework.

For the rest of the section, we fix some natural number $n$.
\begin{cor}
If $(K,v)$ is an n-dependent valued field of positive characteristic $p$, then the value group of $K$ is $p$-divisible.
\end{cor}

Together with Corollary \ref{cor_pnotdiv}, we can conclude the following analogue to \cite[Corollary 5.10]{ik_ts_fow}.
\begin{cor}
Every n-dependent valued field of positive characteristic $p$ whose residue field is perfect, is Kaplansky, i.e. \begin{itemize}
\item the value group is $p$-divisible;
\item the residue field is perfect and does not admit a finite separable extension whose degree is divisible by $p$.
\end{itemize}
\end{cor}

Now, we turn to the question whether an $n$-dependent henselian valued field can carry a nontrivial definable henselian valuation. Note that by a definable henselian valuation $v$ on $K$ we mean that the valuation ring of $(K,v)$, i.\ e.\ the set of elements of $K$ with non-negative value, is a definable set in the language of rings. We need the following definition:

\begin{defn}
Let $K$ be a field. We say that its absolute Galois group is \emph{universal} if for every finite group $G$ there exist a finite extensions $L$ of $K$ and a Galois extension $M$ of $L$  such that $\Gal(M/L) \cong G$.
\end{defn}

As any finite extension of an $n$-dependent field $K$ of characteristic $p>0$ is still $n$-dependent and of characteristic $p$, one cannot find a finite extensions $L \subseteq M$ of $K$ such that their Galois group $\Gal(M/L)$ is of order $p$. Hence any $n$-dependent field of positive characteristic has a non-universal absolute Galois group. Note that Jahnke and Koenigsmann showed in \cite[Theorem 3.15]{fjjk} that a henselian valued field whose absolute value group is non universal and which is neither separably nor real closed admits a non-trivial definable henselian valuation. Hence this gives the following result which is a generalization of \cite[Corollary 3.18]{fjjk}:

\begin{prop}
Let $(K,v)$ be a non-trivially henselian valued field of positive characteristic $p$ which is not separably closed. If $K$ is $n$-dependent then $K$ admits a non-trivial definable henselian valuation.
\end{prop}

\bibliographystyle[

\end{document}